\documentclass[12pt,reqno]{amsart}
\usepackage{amsmath,amsopn,amssymb,amsthm,multicol}
\usepackage{hyperref}
\usepackage{graphicx}

\newcommand{\fr}{\mathfrak}

\DeclareMathOperator{\SO}{SO}

\DeclareMathOperator{\U}{U}

 \newtheorem{lemma} {Lemma} [section]
\newtheorem{theorem}[lemma]{Theorem} 
\newtheorem{remark}[lemma] {Remark} 
\newtheorem{prop} [lemma]{Proposition}  
\newtheorem{definition}[lemma] {Definition} 
\newtheorem{corol}[lemma] {Corollary} 
\newtheorem{example}[lemma] {Example}

\begin{document}

\title{Two-step homogeneous geodesics in homogeneous spaces}
\author{Andreas Arvanitoyeorgos  and Nikolaos Panagiotis Souris}
\address{University of Patras, Department of Mathematics, GR-26500 Patras, Greece}
\email{arvanito@math.upatras.gr}
\email{ntgeva@hotmail.com}
\begin{abstract}

We study geodesics of the form $\gamma(t)=\pi(\exp(tX)\exp(tY))$, $X,Y\in \fr{g}=\operatorname{Lie}(G)$, in homogeneous spaces $G/K$, where $\pi:G\rightarrow G/K$ is the natural projection.  These curves naturally generalise  homogeneous geodesics, that is orbits of one-parameter subgroups of $G$ (i.e. $\gamma(t)=\pi(\exp (tX))$, $X\in \fr{g}$).  We obtain sufficient conditions on a homogeneous space implying the existence of such geodesics for $X,Y\in \fr{m}=T_o(G/K)$.  We use these conditions to obtain examples of Riemannian homogeneous spaces $G/K$ so that all geodesics of $G/K$ are of the above form.  These include total spaces of homogeneous Riemannian submersions endowed with one parameter families of fiber bundle metrics, Lie groups endowed with special one parameter families of left-invariant metrics, generalised Wallach spaces, generalized flag manifolds, and $k$-symmetric spaces with $k$-even,  equipped with certain one-parameter families of invariant metrics.   


\end{abstract}

\maketitle

\medskip

\section{Introduction}

The investigation of geodesics in homogeneous spaces $G/K$ naturally focuses to those geodesics which are  orbits of one parameter subgroups of $G$.  These are curves of the form 
\begin{equation}\label{intro}\gamma(t)=\pi(\exp(tX)),\end{equation}
where $X$ is a non-zero vector in the Lie algebra $\fr{g}$ of $G$ and $\pi$ denotes the natural projection $G\rightarrow G/K$.
There are several classes of homogeneous spaces $G/K$ such that any geodesic $\gamma$ passing through the origin of $G/K$ is such an orbit, including \emph{Lie groups with bi-invariant metrics}, \emph{symmetric spaces}, \emph{compact homogeneous spaces with the standard metric}, and \emph{naturally reductive spaces}.\\
A geodesic of the form (\ref{intro}) is called a \emph{homogeneous geodesic} and a space $G/K$ such that any geodesic of $G/K$ passing through the origin is homogeneous, is called a \emph{geodesic orbit space (g.o. space)}.  Homogeneous geodesics were originally studied by B. Kostant and E.B. Vinberg.  In
\cite{Ko-Va} O. Kowalski and L. Vanhecke initiated the study of g.o spaces in the mid 80's and since then, g.o Riemannian and pseudo-Riemannian spaces have been extensivelly investigated by many authors.  We refer the reader to \cite{Du} for a survey on homogeneous geodesics and g.o. spaces, concerning both the Riemannian and pseudo-Riemannian case.

In this paper we consider a generalisation of homogeneous geodesics, namely geodesics of the form 
\begin{equation}\label{intro1}\gamma(t)=\pi(\exp(tX)\exp(tY)), \quad X,Y\in \fr{g},\end{equation}
which we name \emph{two-step homogeneous geodesics}.  In the main Theorem \ref{maintheorem} we find sufficient conditions on a Riemannian homogeneous space $G/K$, which  imply the existence of two-step homogeneous geodesics in $G/K$.  We use Theorem \ref{maintheorem} to obtain various classes of Riemannian homogeneous spaces $G/K$ such that any geodesic of $G/K$ passing through the origin is two-step homogeneous.   We call these spaces \emph{two-step g.o. spaces}.  A notable class of two-step g.o. spaces are total spaces $G/K$ of a homogeneous fibration
\begin{equation*}H/K\rightarrow G/K\rightarrow G/H
\end{equation*}
endowed with a standard metric which is ``deformed" along the fibers $H/K$ (cf. Proposition \ref{sub}).  Well known examples of such Riemannian spaces are the odd dimensional spheres $\mathbb S^{2n+1}$ as total spaces of the \emph{Hopf fibration}, endowed with \emph{Cheeger deformation} metrics $g_{1+\epsilon}$, $\epsilon>0$.  Other examples of two-step g.o spaces include Lie groups with special one parameter families of left invariant metrics, generalized flag manifolds and generalised symmetric spaces of even order (cf. Section 6).

Geodesics of the form (\ref{intro1}) were introduced by H. C. Wang in \cite{Wa} as geodesics in a semisimple Lie group $G$, equipped with a metric induced by a Cartan involution of the Lie algebra $\fr{g}$ of $G$.  In \cite{Da-Zi} D'Atri and Ziller proved that a large class of left-invariant metrics in a compact Lie group $G$ induce geodesics of the form $\gamma(t)=\exp(tX)\exp(tY)$.  If $G$ is simple, these are precisely the left-invariant metrics which are naturally reductive with respect to a subgroup of the isometry group of $G$.  In \cite{Do} R. Dohira proved that if the tangent space $T_o(G/K)$ of a homogeneous space splits into submodules $\fr{m}_1,\fr{m}_2$ satisfying certain algebraic relations, and if $G/K$ is endowed with a special one parameter family of Riemannian metrics $g_c$, then all geodesics of the Riemannian space $(G/K,g_c)$ are of the form (\ref{intro1}).   Corollary \ref{criterion} 
in the present paper is a generalisation of Dohira's result.
In \cite{Ar-So} the authors studied metrics whose geodesics are of the form 
$\gamma(t)=\pi(\exp(tX)\exp(tY)\exp(tZ))$ in \emph{generalized Wallach spaces}.  Moreover, the authors obtained a necessary and sufficient condition for a curve of the form $\gamma(t)=\pi(\exp(tX)\exp(tY)\exp(tZ))$, $X,Y,Z\in \fr{m}=T_o(G/K)$, to be a geodesic, which is a generalised version of the ``geodesic lemma" for homogeneous geodesics in \cite{Ko-Va}.  

The paper is organised as follows:
In Section 2 we state our main results (Theorem \ref{maintheorem}, Corollary \ref{criterion} and Proposition \ref{sub}).
In Section 3 we give some preliminary facts about invariant metrics on homogeneous spaces $G/K$ and invariant decompositions of $T_o(G/K)$.
In Section 4 we prove Theorem \ref{maintheorem} and Corollary \ref{criterion}.
In Section 5 we prove Proposition \ref{sub} and we give examples of two-step g.o. spaces that are total spaces of homogeneous Riemannian fibrations.  Finally, in Section 6 we use Corollary \ref{criterion} to give examples of Riemannian two-step g.o spaces including Lie groups, flag manifolds and $k$-symmetric spaces. \\

\noindent
{\bf Acknowledgements.}
The work was supported by Grant $\#E.037$ from the Research Committee of the University of Patras (Programme K. Karatheodori).
The first author had fruitful discussions with Professor Hiroshi Tamaru during his visit in Hiroshima University in Spring 2015.  Both authors acknowledge a comment of Professor Yu.G. Nikonorov (cf. Remark 5.4).

\section{Statement of main results}

Let $(G/K,g)$ be a Riemannian homogeneous manifold and consider the natural map $\pi:G\rightarrow G/K$.  Let $o=\pi(e)$ be the origin of $G/K$.   
			
\begin{definition}  A two-step homogeneous geodesic on $G/K$ is a geodesic $\gamma$ with the property that there exist  $X,Y\in \fr{g}$ such that $\gamma(t)=\pi(\exp(tX)\exp(tY))$, for any $t\in \mathbb R$ .\end{definition}

\begin{definition}A two-step geodesic orbit space (two-step g.o. space) is a Riemannian homogeneous space so that all geodesics $\gamma$ with $\gamma(0)=o$, are two-step homogeneous.\end{definition} 

The main theorem is the following: 

\begin{theorem}\label{maintheorem}
 Let $M=G/K$ be a homogeneous space admitting a naturally reductive Riemannian metric.  Let $B$ be the corresponding inner product on $\fr{m}=T_o(G/K)$.  We assume that $\fr{m}$ admits an $\operatorname{Ad}(K)$-invariant orthogonal decomposition 

\begin{equation}\label{decco}\fr{m}=\fr{m}_1\oplus \fr{m}_2\oplus \cdots \oplus \fr{m}_s,\end{equation}

\noindent with respect to $B$.  We equip $G/K$ with a $G$-invariant Riemannian metric $g$ corresponding to the $\operatorname{Ad}(K)$-invariant positive definite inner product 

\begin{equation}\label{metrics}\langle \ ,\ \rangle=\lambda_1\left.B\right|_{\fr{m}_1}+\cdots +\lambda_s\left.B\right|_{\fr{m}_s}, \quad \lambda_1,\cdots,\lambda_s>0.
\end{equation}

\noindent
 If $(\fr{m}_a,\fr{m}_b)$ is a pair of submodules in the decomposition (\ref{decco}) such that 

\begin{equation}\label{geod}[\fr{m}_a,\fr{m}_b]\subset \fr{m}_a,\end{equation}

\noindent then any geodesic $\gamma$ of $(G/K,g)$ with $\gamma(0)=o$ and $\dot{\gamma}(0)\in \fr{m}_a\oplus \fr{m}_b$, is a two-step homogeneous geodesic.\\
In particular, if $\dot{\gamma}(0)=X_a+X_b \in \fr{m}_a\oplus \fr{m}_b$, then for every $t\in \mathbb R$ this geodesic is given by \\

\begin{equation}\label{2step}\gamma(t)=\pi(\exp t(X_a+\lambda X_b)\exp t(1-\lambda)X_b),
\end{equation}

\noindent where $\lambda=\frac{\lambda_b}{\lambda_a}$.\\

Moreover, if any of the following relations holds:\\

1)\quad $\lambda_a=\lambda_b$ or\\

2)\quad  $[\fr{m}_a,\fr{m}_b]=\left\{ {0} \right\}$,\\

then $\gamma$ is a homogeneous geodesic, that is $\gamma(t)=\pi(\exp t(X_a+X_b))$.
\end{theorem}
 
 \medskip
 We note that conditions 1) and 2) are only sufficient conditions for a two-step homogeneous geodesic 
 (\ref{2step}) to be a homogeneous geodesic.
For example, the weakly symmetric flag manifold $\SO(2l+1)/\U(l)$ with tangent space
decomposition
$T_o(\SO(2l+1)/\U(l))=\fr{m}_1\oplus \fr{m}_2$, with respect to the negative of the Killing form of $\fr{so}(2l+1)$, satisfies
$[\fr{m}_1,\fr{m}_2]\subset \fr{m}_1$.
Still, the  $\SO(2l+1)$-metrics given by
$g_{\lambda}=B|_{\fr{m}_1}+\lambda B_{\fr{m}_2}$, $\lambda>0$
are  one-parameter family of  metrics  such that all geodesics  are homogeneous (\cite{Al-Ar}).  
   In this case, neither of the conditions 1), 2) are satisfied.
     Any geodesic $\gamma$ with $\gamma(0)=o$ and $\dot{\gamma}(0)=X_1+X_2$, $X_i\in \fr{m}_i$, $i=1,2$, can be expressed as both  homogeneous and  two-step homogeneous geodesic.
     Indeed,
the homogeneous form is given by $\gamma(t)=\pi(\exp(a+X_1+X_2))$, where $a\in\fr{u}(l)$ 
is a vector   which depends on the choice of $X_1, X_2$.
On the other hand, by  Theorem \ref{maintheorem} $\gamma$ is also given by $\gamma(t)=\pi(\exp t(X_1+\lambda X_2)\exp t(1-\lambda)X_2)$ which is two-step homogeneous.
To find necessary conditions such that a two-step homogeneous geodesic is a (one-step) homogeneous, is  an open problem.

 \medskip
The following corollary provides a method to obtain many examples of two-step g.o. spaces.

\begin{corol}\label{criterion}Let $M=G/K$ be a homogeneous space admitting a naturally reductive Riemannian metric.  Let $B$ be the corresponding inner product of $\fr{m}=T_o(G/K)$.  We assume that $\fr{m}$ admits an $\operatorname{Ad}(K)$-invariant orthogonal decomposition

\begin{equation}\label{22}\fr{m}=\fr{m}_1\oplus\fr{m}_2\end{equation}

\noindent with respect to $B$, such that 
\begin{equation*}[\fr{m}_1,\fr{m}_2]\subset \fr{m}_1.\end{equation*}  Then $M$ admits an one-parameter family of $G$-invariant Riemannian metrics $g_{\lambda}$, $\lambda \in \mathbb R^+$, such that $(M,g_{\lambda})$ is a two-step g.o. space.\\
Each metric $g_{\lambda}$ corresponds to an $\operatorname{Ad}(K)$-invariant positive definite inner product on $\fr{m}$ of the form 

\begin{equation}\label{new}\langle \ ,\ \rangle=\left.B\right|_{\fr{m}_1}+\lambda \left.B\right|_{\fr{m}_2}.\end{equation}

This is homothetic to a metric corresponding to the inner product

\begin{equation*}\langle \ ,\ \rangle=\lambda_1\left.B\right|_{\fr{m}_1}+\lambda_2\left.B\right|_{\fr{m}_2},\end{equation*}

\noindent where $\lambda=\frac{\lambda_2}{\lambda_1}$.\end{corol}

An important class of two-step g.o. spaces can be obtained as the total space of a homogeneous Riemannian submersion $\pi:G/K\rightarrow G/H$, endowed with a special one-parameter family of fiber bundle metrics, as shown below.

\begin{prop}\label{sub}Let $G$ be a Lie group admitting a bi-invariant Riemannian metric and let $K,H$ be closed and connected subgroups of $G$, such that $K\subset H\subset G$.  Let $B$ be the $\operatorname{Ad}$-invariant positive definite inner product on the Lie algebra $\fr{g}$ corresponding to the bi-invariant metric of $G$.  We identify each of the spaces $T_o(G/K),T_o(G/H)$ and $T_o(H/K)$ with corresponding subspaces $\fr{m},\fr{m}_1$ and $\fr{m}_2$ of $\fr{g}$, such that $\fr{m}=\fr{m}_1\oplus \fr{m}_2$.  We endow $G/K$ with the $G$-invariant Riemannian metric $g_{\lambda}$ corresponding to the $\operatorname{Ad}(K)$-invariant positive definite inner product  

\begin{equation*}\langle \ ,\ \rangle=\left.B\right|_{\fr{m}_1}+\lambda \left.B\right|_{\fr{m}_2}, \quad \lambda>0,\end{equation*}

\noindent on $\fr{m}$.  Then $(G/K,g_{\lambda})$ is a two-step g.o. space.

\end{prop}

Note that the natural map $\pi:G/K\rightarrow G/H$ is a Riemannian submersion with totally geodesic fibers (\cite{Ber}).\\

Theorem \ref{maintheorem} and Corollary \ref{criterion} will be proved in Section 4, and Proposition \ref{sub} will be proved in Section 5.  Other applications of Corollary \ref{criterion} will  be given in Section 6.

\section{Preliminaries}

\subsection{Reductive homogeneous Riemannian spaces and $G$-invariant metrics}

Let $G$ be a Lie group with Lie algebra $\fr{g}$, and $K$ be a closed subgroup of $G$.  Consider the homogeneous manifold $G/K$ with origin $o$, and the projection $\pi:G\rightarrow G/K$.  The \emph{left translation} $\tau_g:G/K\rightarrow G/K$ is given by $\tau_g(\pi(h))=\pi(gh)$, $g,h\in G$.  A \emph{$G$-invariant Riemannian metric} on $G/K$ is a Riemannian metric which is invariant under left translations on $G/K$.  A \emph{Riemannian homogeneous space $(G/K,g)$} is a homogeneous manifold $G/K$ endowed with a Riemannian $G$-invariant metric $g$. 

Let $\fr{g},\fr{k}$ be the Lie algebras of $G,K$ respectively and let $\operatorname{Ad}:G\rightarrow \operatorname{Aut}(\fr{g})$ be the adjoint representation of $G$.  The homogeneous space $G/K$ is called \emph{reductive} if there exists a decomposition $\fr{g}=\fr{k}\oplus \fr{m}$ with $\operatorname{Ad}(K)\fr{m}\subset \fr{m}$.  The subspace $\fr{m}$ is naturally identified with the \emph{tangent space} $T_o(G/K)$ through the pushforward $(\pi_*)_e$ of the projection $\pi:G\rightarrow G/K$.  Any $G$-invariant Riemannian metric on $G/K$ corresponds to a unique \emph{$\operatorname{Ad}(K)$-invariant positive definite inner product} on $\fr{m}$ and vice versa.\\

\begin{definition}A reductive Riemannian homogeneous space $(G/K,g)$ is called \emph{naturally reductive} if the $G$-invariant metric $g$ corresponds to a positive definite inner product $B:\fr{m}\times \fr{m}\rightarrow \fr{m}$ with the property 

\begin{equation}\label{nr}B([X,Y]_{\fr{m}},Z)+B(Y,[X,Z]_{\fr{m}})=0, \quad \makebox{for all} \quad X,Y,Z\in \fr{m}.\end{equation}  

\end{definition}

The  class of naturally reductive Riemannian homogeneous spaces includes \emph{symmetric spaces}, \emph{Lie groups with bi-invariant metrics}, \emph{flag manifolds with the standard metric}, and \emph{generalised symmetric spaces}.  

\subsection{Homogeneous geodesics in homogeneous spaces}

It is  well known  that any homogeneous Riemannian manifold $(G/K,g)$ is geodesically complete, i.e. any geodesic $\gamma$ on $G/K$ is defined for all $t\in \mathbb R$.

\begin{definition} 
 
A homogeneous geodesic on a Riemannian homogeneous space $G/K$ is a geodesic $\gamma$ with the property that there exists a vector $X\in \fr{g}\setminus \left\{ {0} \right\}$, such that $\gamma(t)=\pi(\exp(tX))$ for any $t\in \mathbb R$.  
A Riemannian geodesic orbit space (\emph{g.o. space}) is a Riemannian homogeneous space $(G/K,g)$ so that all geodesics $\gamma$ with $\gamma(0)=o$ are homogeneous.\end{definition}

Every naturally reductive space is a g.o space.  Important examples of non naturally reductive g.o. spaces are the \emph{generalised Heisenberg groups}.

\subsection{Invariant decompositions of $\fr{m}=T_o(G/K)$}
     
We assume that $B:\fr{m}\times \fr{m}\rightarrow \fr{m}$ is an $\operatorname{Ad}(K)$-invariant positive definite inner product on $\fr{m}$.  An \emph{$\operatorname{Ad}(K)$-invariant and $B$-orthogonal decomposition of $\fr{m}$} is a decomposition of the form 

\begin{equation*}\fr{m}=\fr{m}_1\oplus \fr{m}_2\oplus \cdots \oplus \fr{m}_s,\end{equation*}

\noindent with respect to $B$, such that $\operatorname{Ad}(K)\fr{m}_i\subset \fr{m}_i$, $i=1,2,\dots,s.$  

  An important example of such a decomposition is the following:

\begin{example}
Let $G/K$ be a compact homogeneous space.  The isotropy representation  $\rho:K\rightarrow \operatorname{Aut}(\fr{m})$ of $G/K$, given by $\rho(k)X=((\tau_k)_*)_o(X)$, $k\in K,X\in \fr{m}$, is completely reducible, therefore it induces a decomposition of $\fr{m}$

\begin{equation}\label{isotropy}\fr{m}=\fr{n}_1\oplus \cdots \oplus \fr{n}_l,\end{equation}

\noindent into irreducible submodules $\fr{n}_i$, $i=1,\dots ,l$ with respect to an $\operatorname{Ad}(K)$-invariant positive definite inner product $B$ on $\fr{m}$ .
It is well known that the isotropy representation $\rho$ of $K$ is equivalent to the restriction $\left.\operatorname{Ad}(K)\right|_{\fr{m}}$ of the adjoint representation of $K$ to $\fr{m}$.  Thus, it is $\operatorname{Ad}(K)\fr{n}_i\subset \fr{n}_i$, $i=1,\dots ,l$.\\
In the decomposition (\ref{isotropy}) there might exist submodules $\fr{n}_j$ that are pairwise equivalent (as subrepresentations of $K$).  If $\fr{n}_a,\fr{n}_b$ are two such equivalent submodules, then $\fr{n}_a$ is not necessarily orthogonal to $\fr{n}_b$ with respect to $B$.  By regrouping the pairwise equivalent subrepresentations $\fr{n}_{j}$ into submodules $\fr{m}_i$, we obtain a decomposition of the form  

\begin{equation*}\fr{m}=\fr{m}_1\oplus \fr{m}_2\oplus \cdots \oplus \fr{m}_s, \quad s\leq l,\end{equation*}

\noindent which is $\operatorname{Ad}(K)$-invariant and orthogonal with respect to $B$.
\end{example}  

If $\fr{m}=\fr{m}_1\oplus \cdots \oplus \fr{m}_s$ is an $\operatorname{Ad}(K)$-invariant, orthogonal decomposition of $\fr{m}$ with respect to an $\operatorname{Ad}(K)$-invariant positive definite inner product $B$, then the inner product 

\begin{equation*}\langle \ ,\ \rangle=\lambda_1\left.B\right|_{\fr{m}_1}+\cdots +\lambda_r\left.B\right|_{\fr{m}_r}, \quad \lambda_1,\cdots,\lambda_s>0, \end{equation*}

\noindent is $\operatorname{Ad}(K)$-invariant and positive definite, therefore it corresponds to a unique $G$-invariant Riemannian metric on $G/K$. 

\begin{remark}\label{remark}Let $\operatorname{ad}:\fr{g}\rightarrow \operatorname{End}(\fr{g})$ be the adjoint representation of $\fr{g}$ given by $\operatorname{ad}(X)Y=[X,Y]$.  The $\operatorname{Ad}(K)$-invariance of a subspace $\fr{m}_i$ implies the $\operatorname{ad}(\fr{k})$-invariance of $\fr{m}_i$.  Moreover, the $\operatorname{Ad}(K)$-invariance of an inner product $B$ implies the $\operatorname{ad}(\fr{k})$-skew symmetry of $B$.  The converse of these statements is true if $K$ is connected.
\end{remark} 

The following example will be used in Sections 5 and 6.

\begin{example}\label{bi} Let $G$ be a Lie group admitting a bi-invariant Riemannian metric $g$.  Let $K$ be a connected subgroup of $G$ and let $\fr{g},\fr{k}$ be the Lie algebras of $G,K$ respectively.  The bi-invariant metric $g$ corresponds to an $\operatorname{Ad}$-invariant positive definite inner product $B$ on $\fr{g}$, which induces an orthogonal decomposition $\fr{g}=\fr{k}\oplus \fr{m}$.  Then this decomposition is $\operatorname{Ad}(K)$-invariant. Indeed, for any $X_{\fr{k}},Y_{\fr{k}}\in \fr{k}$ and $X_{\fr{m}}\in \fr{m}$, we set $Z_{\fr{k}}=[X_{\fr{k}},Y_{\fr{k}}]\in \fr{k}$.  Then

\begin{equation*}B([X_{\fr{k}},X_{\fr{m}}],Y_{\fr{k}})=-B(X_{\fr{m}},[X_{\fr{k}},Y_{\fr{k}}])=-B(X_{\fr{m}},Z_{\fr{k}})=0.\end{equation*}    

It follows that $[X_{\fr{k}},X_{\fr{m}}]\in \fr{m}$, and since $K$ is connected then $\operatorname{Ad}(K)\fr{m}\subset \fr{m}$ (cf. Remark \ref{remark}).\end{example}

\subsection{Local projections of vector fields}

Let $\pi: G\to G/K$ be the projection and $p\in G$.
Then for each vector field $V$ in $G$ there exists an open neighborhood $U_{\pi(p)}$ of $\pi(p)$ in $G/K$, such that $\pi_*V$ is a well defined vector field in $U_{\pi(p)}$.  Indeed, since $\pi:G\rightarrow G/K$ is a submersion, there exists a neighborhood $U$ of $e$ in $G$ such that $\pi|_{U}:U\rightarrow \pi(U)$ is a bijection.  Moreover, $\pi(U)$ is an open neighborhood of $o$ in $G/K$ (cf. \cite[p. 546]{Lee}).  We set $U_{\pi(p)}=\tau_p(\pi(U))=\pi(L_pU)$, where $L_p$ is the left translation in $G$.  Since $L_p$ is a diffeomorphism of $G$, the map $\pi|_{L_pU}:L_pU\rightarrow U_{\pi(p)}$ is a bijection, therefore $\pi_*V$ is a well defined vector field in $U_{\pi(p)}$.

\section{Proof of the main results}

In this section we will prove Theorem \ref{maintheorem} and Corollary \ref{criterion}.  
The following lemmas will be useful.
Let $X^R,Y^L$ denote the right-invariant and  left-invariant vector fields in $G$ 
induced by $X,Y$, respectively.

\begin{lemma}\label{left} Let $G/K$ be a homogeneous space and let $X,Y \in \fr{g}$.    Then
\begin{equation}\label{l2}[X^L,Y^R]=0.\end{equation}
\end{lemma}

\begin{proof}    
Let $p\in G$ and let $f:G\rightarrow \mathbb R$ be a smooth function.  It is

\begin{equation*}
\begin{array}{lll}
 \displaystyle{[X^L,Y^R]_pf} &=&\displaystyle{X^L_p(Y^Rf)}-\displaystyle{{Y^R_p}(X^Lf)} \nonumber \ \\\\

\ \ \ &=&\displaystyle{\left.\frac{d}{dt}\right|_{t=0}(Y^Rf)(p\exp (tX))}-\displaystyle{\left.\frac{d}{dt}\right|_{t=0}(X^Lf)(\exp (tY) p)}\nonumber \ \\\\

\ \ \ &=& \displaystyle{\left.\frac{d}{dt}\right|_{t=0}Y^R_{p\exp (tX)}f}-\displaystyle{\left.\frac{d}{dt}\right|_{t=0}X^L_{\exp (tY) p}f}\nonumber \ \\\\

\ \ \ &=&\displaystyle{\left.\frac{d}{dt}\right|_{t=0}\left.\frac{d}{ds}\right|_{s=0}f(\exp (sY) p \exp (tX))}-\displaystyle{\left.\frac{d}{dt}\right|_{t=0}\left.\frac{d}{ds}\right|_{s=0}f(\exp (tY) p \exp (sX))}=0.
\end{array}
\end{equation*} 

\end{proof}

\begin{lemma}\label{analytic} Let $G$ be a Lie group with Lie algebra $\fr{g}$.  We assume that there exist subspaces $\fr{m}_a,\fr{m}_b$ of $\fr{g}$ such that $[\fr{m}_a,\fr{m}_b]\subset \fr{m}_a$.  Then $\operatorname{Ad}(\exp(\fr{m}_b))\fr{m}_a\subseteq \fr{m}_a$.\end{lemma}

\begin{proof}  

The homomorphism $\operatorname{Ad}:G\rightarrow \operatorname{Aut}{\fr{g}}$ is analytic (\cite{Hel} p.126), and for any $X\in \fr{g}$, it is

\begin{equation*}\operatorname{Ad}(\exp_{\fr{g}}(X))=\exp_{\operatorname{Aut}{\fr{g}}}\operatorname{ad}(X)=\sum_{n=0}^{\infty}{\frac{1}{n!}\operatorname{ad}^n(X)}.\end{equation*}

 Let $X_a\in \fr{m}_a$ and $X_b\in \fr{m}_b$.  Then 
\begin{equation*}\operatorname{Ad}(\exp_{\fr{g}}(X_b))X_a=\sum_{n=0}^{\infty}{\frac{1}{n!}\operatorname{ad}^n(X_b)X_a}.\end{equation*}

  Since $[\fr{m}_a,\fr{m}_b]\subset \fr{m}_a$, we use induction to obtain that $\operatorname{ad}^n(X_b)X_a\in \fr{m}_a$, for any $n\in \mathbb N$.  Consequently, for any $N \in \mathbb N$, we have that 
	
	\begin{equation*}\sum_{n=0}^{N}{\frac{1}{n!}\operatorname{ad}^n(X_b)X_a}\in \fr{m}_a.\end{equation*}  Since $\fr{m}_a$ is (topologically) closed, then 
	
	\begin{equation*}\lim_{N\rightarrow \infty}{\sum_{n=0}^{N}{\frac{1}{n!}\operatorname{ad}^n(X_b)X_a}}\in \fr{m}_a,\end{equation*} 
	which proves the lemma.
\end{proof}

\medskip
\noindent
\emph{Proof of Theorem \ref{maintheorem}}. 
We set $X=X_a+\lambda X_b$, $Y=(1-\lambda)X_b$, $\alpha(t)=\exp tX \exp tY$ and $\gamma=\pi\circ \alpha$.  Let $\nabla$ be the Riemannian connection of $(G/K,g_{\lambda})$.  Then $\gamma (t)$ is a geodesic if and only if $\nabla_{\dot{\gamma}}\dot{\gamma}=0$.  By using Koszul's formula we have that

\begin{equation} \label{koz1} g(V,\nabla_{\dot{\gamma}}\dot{\gamma})=\dot{\gamma}g(V,\dot{\gamma})+g(\dot{\gamma},[V,\dot{\gamma}])-\frac{1}{2}Vg(\dot{\gamma},\dot{\gamma}),\end{equation}

\noindent for any vector field $V$ in $G/K$.  We will first show that the vector field $\nabla_{\dot{\gamma}}\dot{\gamma}$ is well defined.  Indeed, let $R_{\alpha(t)}$, $L_{\alpha(t)}$ be the right and left translation respectively on $G$ by $\alpha(t)$.  Then

\begin{eqnarray}
 \displaystyle{\dot{\alpha}(t)} &=&\displaystyle{\left.\frac{d}{ds}\right|_{s=0}\alpha(t+s)}=\displaystyle{\left.\frac{d}{ds}\right|_{s=0}\exp (t+s)X \exp (t+s)Y} \nonumber \\
 &=&
\displaystyle{\left.\frac{d}{ds}\right|_{s=0}\exp (t+s)X \exp tY} +\displaystyle{\left.\frac{d}{ds}\right|_{s=0}\exp tX \exp (t+s)Y} \nonumber \\
 &=& 
\displaystyle{\left.\frac{d}{ds}\right|_{s=0}\exp sX \exp tX \exp tY} + \displaystyle{\left.\frac{d}{ds}\right|_{s=0}\exp tX \exp tY \exp sY} \nonumber \\
&=&
\displaystyle{\left.\frac{d}{ds}\right|_{s=0}\exp sX \alpha(t)}+ \displaystyle{\left.\frac{d}{ds}\right|_{s=0}\alpha(t) \exp sY} \nonumber \\
&=&
\displaystyle{(R_{\alpha(t)})_*(X)+(L_{\alpha(t)})_*Y}=\displaystyle{(X^R+Y^L)_{\alpha(t)}}.\label{dott}
\end{eqnarray}

\noindent
  Equation (\ref{dott}) implies that the vector field $\dot{\alpha}$ along the curve $\alpha$ can be extended to the vector field $X^R+Y^L$ in $G$.  Then, for any $t\in \mathbb R$, there exists a  neighborhood $U_{\pi(\alpha(t))}$ of $\pi(\alpha(t))=\gamma(t)$ in $G/K$ such that $\pi_*(X^R+Y^L)$ is a well defined vector field in $U_{\pi(\alpha(t))}$, which locally extends $\dot{\gamma}$  (cf. Section 3.4).  Therefore, $\nabla_{\dot{\gamma}}\dot{\gamma}$ is well defined.\\
Next, we will show that the right-hand side of equation (\ref{koz1}) vanishes for any $t\in \mathbb R$ and for any vector field $V$ in $G/K$.   
Since a basis of $\fr{m}$ can be
transfered to  a basis of $T_{\pi(p)}G/K$ by $(\tau_p)_*$, it suffices to consider any vector field $V$ defined by 

\begin{equation}\label{vf}V_{\pi(p)}=(\tau_p)_*Z=(\pi_*Z^L)_{\pi(p)}, \quad Z\in \fr{m}.\end{equation}

Moreover, we have that 

\begin{eqnarray}
\displaystyle{\dot{\gamma}(t)}&=&\displaystyle{(\pi_*)_{\alpha(t)}(X^R+Y^L)_{\alpha(t)}}=\displaystyle{(\tau_{\alpha(t)})_*((\tau_{\alpha(t)^{-1}})_*(\pi_*)X^R_{\alpha(t)}+Y)}\label{uu} \\
&=&
\displaystyle{(\tau_{\alpha(t)})_*(\left.\operatorname{Ad}({\alpha(t)}^{-1})X\right|_{\fr{m}}+Y)}. \label{i}
\end{eqnarray}

We set 
\begin{equation*}TX=\operatorname{Ad}({\alpha(t)}^{-1})X=\operatorname{Ad}(\exp(-tY)\exp (-tX))X=\operatorname{Ad}(\exp(-tY))X.\end{equation*}
Since $Y=(1-\lambda)X_b\in \fr{m}_b$, and by taking into account relation (\ref{geod}),  Lemma \ref{analytic} implies that $TX_a\in \fr{m}_a$.  Note also that $TX_b=X_b$.  We set $Y_a(t)=TX_a$.  Then 

\begin{equation}\label{TT}TX=T(X_a+\lambda X_b)=Y_a(t)+\lambda X_b, \end{equation}
and by relation (\ref{i}) it follows that
\begin{equation}\label{ii}\dot{\gamma}(t)=(\tau_{\alpha(t)})_*(Y_a(t)+X_b).\end{equation}
Let $t\in \mathbb R$. By using equations (\ref{vf}) and (\ref{i}) and by taking into account the $G$-invariance of the metric $g$, the first term in the right-hand side of (\ref{koz1}) becomes 

\begin{eqnarray}
\displaystyle{\dot{\gamma}(t)g(V,\dot{\gamma})}&=&\displaystyle{\left.\frac{d}{ds}\right|_{s=0}g((\tau_{\alpha(t+s)})_*Z,(\tau_{\alpha(t+s)})_*(\left.\operatorname{Ad}({\alpha(t+s)}^{-1})X\right|_{\fr{m}}+Y))}\nonumber \\
&=&
\displaystyle{\left.\frac{d}{ds}\right|_{s=0}\langle Z,\left.\operatorname{Ad}({\alpha(t+s)}^{-1})X\right|_{\fr{m}}+Y\rangle}.\label{t1}     
\end{eqnarray}

Also,
\begin{eqnarray}\displaystyle{\left.\frac{d}{ds}\right|_{s=0}\left.\operatorname{Ad}({\alpha(t+s)}^{-1})X\right|_{\fr{m}}}&=&\displaystyle{\left.\left.\frac{d}{ds}\right|_{s=0}\operatorname{Ad}(\exp (-t-s)Y \exp (-t-s)X)X\right|_{\fr{m}}}\nonumber \\
&=&
\displaystyle{\left.\left.\frac{d}{ds}\right|_{s=0}\operatorname{Ad}(\exp (-t-s)Y)X\right|_{\fr{m}}}\nonumber \\
&=&
\displaystyle{[\operatorname{Ad}(\exp (-tY))X,Y]_{\fr{m}}}=\displaystyle{[TX,Y]_{\fr{m}}}.\label{t}
\end{eqnarray}

By considering relations (\ref{metrics}), (\ref{geod}) and by taking into account relations (\ref{TT}) and (\ref{t}), then equation (\ref{t1}) implies that the first term of the right-hand side of equation (\ref{koz1}) reduces to

\begin{eqnarray}\displaystyle{\dot{\gamma}(t)g(V,\dot{\gamma})}&=&\displaystyle{\langle [Y_a(t)+\lambda X_b,(1-\lambda)X_b]_{\fr{m}},Z\rangle}=\displaystyle{(1-\lambda) \langle [Y_a(t),X_b]_{\fr{m}},Z\rangle}\nonumber \\
&=&\displaystyle{(1-\lambda) \langle [Y_a(t),X_b]_{\fr{m}_a},Z\rangle}=(1-\lambda)\lambda_aB(Z,[Y_a(t),X_b]_{\fr{m}})\label{term1}.
\end{eqnarray}

To calculate the second term in the right hand side of (\ref{koz1}), we use expressions (\ref{uu}), (\ref{ii}) and Lemma \ref{left}, as well as  relation (\ref{vf}) and the $G$-invariance of the metric $g$.  Hence we obtain that

\begin{eqnarray}\displaystyle{g(\dot{\gamma}(t),[V,\dot{\gamma}]_{\gamma(t)})}&=&\displaystyle{g((\tau_{\alpha(t)})_*(Y_a(t)+X_b),[(\pi_*)Z^L,(\pi_*)(X^R+Y^L)]_{\gamma(t)})}\nonumber\\
&=&
\displaystyle{g((\tau_{\alpha(t)})_*(Y_a(t)+X_b),((\pi_*)[Z^L,X^R+Y^L])_{\gamma(t)})}\nonumber \\
&=&
\displaystyle{g((\tau_{\alpha(t)})_*(Y_a(t)+X_b),(\pi_*)(L_{\alpha(t)})_*[Z,Y])}\nonumber \\
&=&
\displaystyle{g((\tau_{\alpha(t)})_*(Y_a(t)+X_b),(\tau_{\alpha(t)})_*(\pi_*)_e[Z,Y])}\nonumber \\
&=&
\displaystyle{\langle Y_a(t)+X_b,[Z,Y]_{\fr{m}}\rangle}=\displaystyle{(1-\lambda)\langle Y_a(t)+X_b,[Z,X_b]_{\fr{m}}\rangle}.\label{nrp} 
\end{eqnarray}

By taking into account relation (\ref{metrics}) and the natural reductivity property of $B$, it follows that

\begin{eqnarray}\displaystyle{\langle Y_a(t)+X_b,[Z,X_b]_{\fr{m}}\rangle}&=&\displaystyle{\lambda_aB( Y_a(t),[Z,X_b]_{\fr{m}}) +\lambda_bB (X_b,[Z,X_b]_{\fr{m}})}\nonumber \\
&=&
\displaystyle{-\lambda_aB(Z,[Y_a(t),X_b]_{\fr{m}})-\lambda_bB(Z,[X_b,X_b]_{\fr{m}})}\nonumber \\
&=&
\displaystyle{-\lambda_aB(Z,[Y_a(t),X_b]_{\fr{m}})}.\label{te}
\end{eqnarray}

By using (\ref{te}), then relation (\ref{nrp}) implies that the second term in the right-hand side of equation (\ref{koz1}) reduces to 

\begin{equation}\label{term2}(\lambda-1)\lambda_aB(Z,[Y_a(t),X_b]_{\fr{m}}).\end{equation}

To calculate the third term of (\ref{koz1}), we use the local extension $\pi_*(X^R+Y^L)$ of $\dot{\gamma}(t)$, which at $\pi(p)$ ($p\in G$) is given by

\begin{equation*}(\pi_*)_p(X^R+Y^L)_p=(\tau_{p^{-1}})_*(\left.\operatorname{Ad}(p^{-1})X\right|_{\fr{m}}+Y).\end{equation*}

We have that

\begin{eqnarray}\displaystyle{V_{\gamma(t)}g(\dot{\gamma},\dot{\gamma})}&=&\displaystyle{((\tau_{\alpha(t)})_*Z)g(\dot{\gamma},\dot{\gamma})}\nonumber \\
&=&
\displaystyle{\left.\frac{d}{ds}\right|_{s=0}g((\tau_{p^{-1}})_*(\left.\operatorname{Ad}(p^{-1})X\right|_{\fr{m}}+Y),(\tau_{p^{-1}})_*(\left.\operatorname{Ad}(p^{-1})X\right|_{\fr{m}}+Y))}\nonumber \\
&=&
\displaystyle{\left.\frac{d}{ds}\right|_{s=0}\langle \left.\operatorname{Ad}(p^{-1})X\right|_{\fr{m}}+Y,\left.\operatorname{Ad}(p^{-1})X\right|_{\fr{m}}+Y\rangle },\label{dom} 
\end{eqnarray}	
	
\noindent where $p=\alpha(t)\exp sZ$.  Notice that $\pi(p)$ is sufficiently close to $\gamma(t)$ if $s$ is sufficiently small, therefore $\pi(p)$ lies in the domain of the local extension $\pi_*(X^R+Y^L)$ for small $s$.  Moreover, it is

\begin{eqnarray}\displaystyle{\left.\frac{d}{ds}\right|_{s=0}\left.\operatorname{Ad}(p^{-1})X\right|_{\fr{m}}}&=&\displaystyle{\left.\left.\frac{d}{ds}\right|_{s=0}\operatorname{Ad}(\exp(-sZ))\operatorname{Ad}({\alpha(t)}^{-1})X\right|_{\fr{m}}}\nonumber \\
&=&
\displaystyle{[TX,Z]_{\fr{m}}}=\displaystyle{[Y_a(t),Z]_{\fr{m}}+\lambda[X_b,Z]_{\fr{m}}}.\label{en}
\end{eqnarray}	
	
By using relation (\ref{en}), assumptions (\ref{metrics}) and (\ref{geod}) of the theorem, and the natural reductivity property of $B$, then equation (\ref{dom}) yields the
third term of the right hand side of (\ref{koz1}) as follows:

\begin{eqnarray}\displaystyle{V_{\gamma(t)}g(\dot{\gamma},\dot{\gamma})}&=&\displaystyle{2\langle [Y_a(t),Z]_{\fr{m}}+\lambda[X_b,Z]_{\fr{m}},Y_a(t)+X_b\rangle}\nonumber \\
&=&
\displaystyle{2\langle [Y_a(t),Z]_{\fr{m}},Y_a(t)\rangle}+\displaystyle{2\langle [Y_a(t),Z]_{\fr{m}},X_b\rangle}\nonumber \\
&&
+\displaystyle{2\lambda\langle [X_b,Z]_{\fr{m}},Y_a(t)\rangle}+\displaystyle{2\lambda\langle [X_b,Z]_{\fr{m}},X_b\rangle}\nonumber \\ 
&=&
\displaystyle{-2\lambda_aB(Z,[Y_a(t),Y_a(t)]_{\fr{m}})}-\displaystyle{2\lambda_bB(Z,[Y_a(t),X_b]_{\fr{m}})}\nonumber \\
&&
+\displaystyle{2\lambda\lambda_aB(Z,[Y_a(t),X_b]_{\fr{m}})}-\displaystyle{2\lambda\lambda_bB(Z,[X_b,X_b]_{\fr{m}})}\nonumber \\ 
&=&
\displaystyle{2B(Z,[Y_a(t),X_b]_{\fr{m}})(\lambda\lambda_a-\lambda_b)}=0\label{term3}.
\end{eqnarray}
	
By summing equations (\ref{term1}), (\ref{term2}) and (\ref{term3}), we obtain that the right-hand side of equation (\ref{koz1}) vanishes, so the first part of the theorem  follows. 	
	
\smallskip

If $\lambda_a=\lambda_b$, then $\lambda=1$, therefore 

\begin{equation*}\gamma(t)=\pi(\exp t(X_a+\lambda X_b)\exp t(1-\lambda)X_b)=\pi(\exp t(X_a+X_b)).\end{equation*}

Finally, if $[\fr{m}_a,\fr{m}_b]=\left\{ {0} \right\}$ then the vectors $X_a+\lambda X_b$ and $X_b$ commute, therefore
 
\begin{equation*}\gamma(t)=\pi(\exp t(X_a+\lambda X_b)\exp t(1-\lambda)X_b)=\pi(\exp [t(X_a+\lambda X_b)+t(1-\lambda)X_b])=\pi(\exp t(X_a+X_b)),\end{equation*}
 
which proves the second part of the theorem. 
\hfill\(\Box\)\\

\noindent
\emph{Proof of Corollary \ref{criterion}}.  Since $[\fr{m}_1,\fr{m}_2]\subset \fr{m}_1$, then Theorem \ref{maintheorem} implies that any geodesic $\gamma$ of $M$ with $\gamma(0)=o$ and $\dot{\gamma}(0)\in \fr{m}_1\oplus \fr{m}_2$ is two-step homogeneous.  Hence all geodesics $\gamma$ of $M$ with $\gamma(0)=o$ are two-step homogeneous, therefore $(M,g_{\lambda})$ is a two-step g.o space.
\hfill\(\Box\)

\section{Total spaces of homogeneous Riemannian submersions}

A natural application of Corollary \ref{criterion} is for total spaces of homogeneous Riemannian submersions.  We have the following:

\begin{prop}\label{su}Let $G$ be a Lie group admitting a bi-invariant Riemannian metric and let $K,H$ be closed and connected subgroups of $G$, such that $K\subset H\subset G$.  Let $B$ be the $\operatorname{Ad}$-invariant positive definite inner product on the Lie algebra $\fr{g}$ corresponding to the bi-invariant metric of $G$.  We identify each of the spaces $T_o(G/K),T_o(G/H)$ and $T_o(H/K)$ with corresponding subspaces $\fr{m},\fr{m}_1$ and $\fr{m}_2$ of $\fr{g}$, such that $\fr{m}=\fr{m}_1\oplus \fr{m}_2$.  We endow $G/K$ with the $G$-invariant Riemannian metric $g_{\lambda}$ corresponding to the $\operatorname{Ad}(K)$-invariant positive definite inner product  

\begin{equation}\label{ta}\langle \ ,\ \rangle=\left.B\right|_{\fr{m}_1}+\lambda \left.B\right|_{\fr{m}_2}, \quad \lambda>0,\end{equation}

\noindent on $\fr{m}$.  Then $(G/K,g_{\lambda})$ is a two-step g.o. space.

\end{prop}

\begin{proof}

Let $\fr{k},\fr{h},\fr{g}$ be the Lie algebras of the groups $K,H,G$ respectively.  The subspaces $\fr{m}_1$ and $\fr{m}_2$ can be obtained by the $B$-orthogonal decompositions

\begin{equation}\label{123}\fr{g}=\fr{h}\oplus \fr{m}_1 \quad \makebox{and} \quad \fr{h}=\fr{k}\oplus \fr{m}_2,\end{equation}

\noindent such that 
\begin{equation}\label{124}\operatorname{Ad}(H)\fr{m}_1\subset \fr{m}_1, \quad \operatorname{Ad}(K)\fr{m}_2\subset \fr{m}_2
\end{equation}

\noindent
(cf. Example \ref{bi}).  The inner product $\left.B\right|_{\fr{m}}$ induces a naturally reductive metric on $G/K$.  Then relations (\ref{124}) and the orthogonality of $\fr{m}_1,\fr{m}_2$ with respect to $B$ imply that the decomposition $\fr{m}=\fr{m}_1\oplus \fr{m}_2$ is $\operatorname{Ad}(K)$-invariant and $B$-orthogonal.  Moreover, since $\operatorname{Ad}(H)\fr{m}_1\subset \fr{m}_1$, we have that $[\fr{m}_1,\fr{h}]\subset \fr{m}_1$, therefore,  

\begin{equation*}[\fr{m}_1,\fr{m}_2]\subset [\fr{m}_1,\fr{h}]\subset \fr{m}_1.\end{equation*}

\noindent
By Corollary \ref{criterion} it follows that $(G/K,g_{\lambda})$ is a two-step g.o. space.\end{proof}

\begin{example} The odd dimensional sphere $\mathbb S^{2n+1}$ can be considered as the total space of the homogeneous Hopf bundle 
\begin{equation}\label{fi}\mathbb S^1 \rightarrow \mathbb S^{2n+1} \rightarrow \mathbb CP^n.\end{equation}
Let $g_1$ be the standard metric of $\mathbb S^{2n+1}$.  We equip $\mathbb S^{2n+1}$ with an one parameter family of metrics $g_{\lambda}$, which ``deform" the standard metric along the Hopf circles $\mathbb S^1$.  \\
By setting $G=U(n+1)$, $K=U(n)$ and $H=U(n)\times U(1)$, the fibration (\ref{fi}) corresponds to the fibration

\begin{equation*}H/K \rightarrow G/K \rightarrow G/H.\end{equation*}

\noindent
Since $U(n+1)$ is compact, it admits a bi-invariant metric corresponding to an $\operatorname{Ad}(U(n+1))$-invariant positive definite inner product $B$ on $\fr{u}(n+1)$.  We identify each of the spaces $T_o\mathbb S^{2n+1}= T_o(G/K),T_o\mathbb CP^n= T_o(G/H)$, and $T_o\mathbb S^1= T_o(H/K)$ with corresponding subspaces $\fr{m},\fr{m}_1$, and $\fr{m}_2$ of $\fr{u}(n+1)$.  The desired 
one-parameter family of metrics $g_{\lambda}$ corresponds to the one-parameter family of positive definite inner products 

\begin{equation}\label{o}\langle \ ,\ \rangle=\left.B\right|_{\fr{m}_1}+\lambda\left.B\right|_{\fr{m}_2}, \quad \lambda>0\end{equation}

\noindent on $\fr{m}=\fr{m}_1\oplus \fr{m}_2$.  Note that for $\lambda=1$ the inner product (\ref{o}) induces the standard metric $g_1$ on $\mathbb S^{2n+1}$.  Then Proposition \ref{su} implies that $(\mathbb S^{2n+1},g_{\lambda})$ is a two-step g.o. space.  In particular, let $X\in T_o\mathbb S^{2n+1}$.  Then the unique geodesic $\gamma$ of  $(\mathbb S^{2n+1},g_{\lambda})$ with $\gamma(0)=o$ and $\dot{\gamma}(0)=X$, is given by

\begin{equation*}\gamma(t)=\pi(\exp t(X_1+\lambda X_2) \exp t(1-\lambda)X_2), \quad t\in \mathbb R,\end{equation*}

\noindent where $X_1$, $X_2$ are the projections of $X$ on $\fr{m}_1= T_o\mathbb CP^n$ and $\fr{m}_2= T_o\mathbb S^1$ respectively.\\
We remark that if $\lambda=1+\epsilon$, $\epsilon >0$, then the metrics $g_{1+\epsilon}$ are Cheeger deformations of the natural metric $g_1$.  The spaces $(\mathbb S^{2n+1},g_{1+\epsilon})$ are examples of Riemannian homogeneous spaces with positive sectional curvature (\cite{Ch}). 

\end{example} 

\begin{remark} There exist total spaces $G/K$ of homogeneous fibrations $H/K\rightarrow G/K\rightarrow G/H$ which are g.o. with respect to the metric (\ref{ta}) given in Proposition \ref{su}.  In \cite{Ta} H. Tamaru obtained a classification of the total spaces $G/K$ with the property that:

\noindent 1) $G/K$ is fibered over an irreducible symmetric space $G/H$, and \\
2) $G/K$ is g.o. with respect to the metric (\ref{ta}).
\end{remark}

\begin{remark} It was brought to the attention of the authors by Professor Yu.G. Nikonorov, that if $G$ is compact and semisimple, then it is possible to obtain a class of two-step g.o. metrics of $G/K$ as follows. 
 We can consider the action of $G\times L$ on $G/K$, where $L$ is the normaliser of $K$ in $G$.  If $\fr{l}$ is the Lie algebra of $L$, then there exists an $\operatorname{Ad}(K)$-invariant decomposition $\fr{g}=\fr{l}\oplus \fr{m}_1=\fr{k}\oplus \fr{m}_2\oplus \fr{m}_1$, with respect to the negative of the Killing form of $\fr{g}$ (here denoted by $B$).
  We endow $G/K$ with the $G$-invariant metric $g_{\lambda}$ as given in (\ref{new}).  
  Then by using an embedding of $\fr{m}$ into $\fr{g}\times \fr{l}$ as pointed in \cite[p. 585-586]{Zi}, one may lift the metric $g_{\lambda}$ to a naturally reductive metric $\widehat{g}_{\lambda}$ in the space $(G\times L)/\widehat{K}$, where $\widehat{K}=\left\{ {(a,b)\in G\times L:ab^{-1}\in K} \right\}$.  The Riemannian homogeneous space $((G\times L)/\widehat{K},\widehat{g}_{\lambda})$ is a g.o. space, hence the geodesics of $((G\times L)/\widehat{K},\widehat{g}_{\lambda})$ are orbits of the one parameter subgroups $(\exp (tX), \exp (-tY))$ of $G\times L$.  By using the right action of $L$ on $G/K$, the geodesics $\pi(\exp (tX), \exp (-tY))$ in $(G\times L)/\widehat{K}$ correspond to two-step homogeneous geodesics of the form $\pi(\exp tX \exp tY)$ in $G/K$.
\end{remark}

\section{Further examples of two-step g.o. spaces}

In the present section we will use Corollary \ref{criterion} to construct various classes of two-step g.o spaces.  The recipe is the following:\\

$\bullet$ Let $G/K$ be  a homogeneous space with reductive decomposition $\fr{g}=\fr{k}\oplus\fr{m}$ admitting a naturally reductive metric corresponding to a positive definite inner product $B$ on $\fr{m}$.\\
  
$\bullet$ We consider an $\operatorname{Ad}(K)$-invariant, orthogonal decomposition 
$\fr{m}=\fr{n}_1\oplus \cdots \oplus \fr{n}_s$ with respect to $B$.\\  

$\bullet$ We separate the submodules $\fr{n}_i$ into two groups as

\begin{equation*}\fr{m}_1=\fr{n}_{i_1}\oplus \cdots \oplus \fr{n}_{i_n} \quad \makebox{and} \quad \fr{m}_2=\fr{n}_{i_{n+1}}\oplus \cdots \oplus \fr{n}_{i_s},\end{equation*}

\noindent so that $[\fr{m}_1,\fr{m}_2]\subset \fr{m}_1$.\\ 

$\bullet$ We consider the decomposition $\fr{m}=\fr{m}_1\oplus \fr{m}_2$, which  is $\operatorname{Ad}(K)$-invariant and orthogonal with respect to $B$.  Then  Corollary \ref{criterion} implies that $G/K$ admits an one parameter family of metrics $g_{\lambda}$, so that $(G/K,g_{\lambda})$ is a two-step g.o. space.\\

\noindent
 We will apply the above recipe to the following classes of homogeneous spaces:

\noindent
1)\quad Lie groups with bi-invariant metrics equipped with an one-parameter family of left-invariant metrics.

\noindent
2)\quad Flag manifolds equipped with certain one-parameter families of diagonal metrics.

\noindent
3)\quad Generalized Wallach spaces equipped with three different types of diagonal metrics.

\noindent
4)\quad $k$-symmetric spaces where $k$ is even, endowed with an one parameter family of diagonal metrics.\\

 \subsection{Lie groups}

Let $G$ be a Lie group admitting a bi-invariant Riemannian metric and let $B$ be the corresponding $\operatorname{Ad}$-invariant positive definite inner product on its Lie algebra $\fr{g}$.  We consider a subgroup $K$ of $G$ with Lie algebra $\fr{k}$.  The subgroup $K$ induces an $\operatorname{Ad}$-invariant and orthogonal decomposition 
$\fr{g}=\fr{k}\oplus \fr{m}$
with respect to $B$, such that 
\begin{equation}\label{m}[\fr{k},\fr{m}]\subset \fr{m}\end{equation}

\noindent
(cf. Example \ref{bi}).  We view $G$ as the homogeneous space $G/\left\{ {e} \right\}$ and we endow $G$ with the left invariant metric $g_{\lambda}$ corresponding to the positive definite inner product 

\begin{equation}\langle \ ,\ \rangle=\left.B\right|_{\fr{m}}+\lambda \left.B\right|_{\fr{k}}, \quad \lambda>0.\end{equation}

By taking into account relation (\ref{m}) and by using Corollary \ref{criterion} we conclude that $(G,g_{\lambda})$ is a two-step g.o. space.

\subsection{Flag manifolds} 

A \emph{generalized flag manifold} is a homogeneous manifold $G/K$ where $G$ is a compact, connected and semisimple Lie group, and $K$ is the centralizer of a torus $T$ in $G$.  Every generalized flag manifold $G/K$ is a product of generalized flag manifolds $G_i/K$ of simple Lie groups $G_i$.  Therefore, to study flag manifolds it suffices to consider flag manifolds $G/K$, where $G$ is a compact, connected and simple Lie group.  The classification of the g.o. Riemannian flag manifolds $(G/K,g)$, where  $g$ is a non-standard $G$-invariant metric, was obtained in \cite{Al-Ar}. 
Let $G/K$ be a generalized flag manifold with $G$ simple.  Let $B$ be the negative of the Killing form of $\fr{g}$.  Then $B$ is an $\operatorname{Ad}$-invariant positive definite inner product on $\fr{g}$, which induces an orthogonal decomposition $\fr{g}=\fr{k}\oplus \fr{m}$, with $\fr{m}$ naturally identified with $T_o(G/K)$. 
 
We will briefly describe the structure of $\fr{m}$.  For a more detailed descripition, we refer  to \cite{Ar-Ch-Sa} or \cite{Bo-Fo-Ro}.
Consider the complexified Lie algebra $\fr{g}^{\mathbb C}$ and a Cartan subalgebra $\fr{h}^{\mathbb C}$ of $\fr{g}^{\mathbb C}$ which is the complexification of a  maximal abelian subalgebra $\fr{h}$ of $\fr{g}$.  Let $\Pi=\left\{ {\alpha_1,\dots ,\alpha_{r},\alpha_{r+1},\dots ,\alpha_{\ell}}\right\}$ be a set of simple roots of the root system $R$ of $\fr{g}^{\mathbb C}$, with respect to $\fr{h}^{\mathbb C}$.  The flag manifold $M=G/K$ is determined by a subset $\Pi_M$ of $\Pi$, say $\Pi_M=\left\{ {\alpha_1,\dots ,\alpha_{r} }\right\}$.  We call $\Pi_M$ the \emph{set of simple complementary roots of $G/K$}.  Let $\Pi_K=\Pi \setminus \Pi_M$ and define $R_K= \operatorname{span}_{\mathbb Z}\left\{ {\alpha:\alpha\in \Pi_K } \right\}\cap R$, $R_M=R\setminus R_K$ and $R_M^+=\left\{ {\alpha\in R_M :\alpha>0} \right\}$, $R_K^+=\left\{ {\alpha\in R_K :\alpha>0} \right\}$.  We choose a Weyl basis $\left\{ {E_{\alpha}:\alpha\in R } \right\}$ of $\fr{g}^{\mathbb C}$, and
let $\alpha\in R^+$.  We set 
\begin{equation*}A_\alpha=E_{\alpha}+E_{-\alpha},\quad  B_\alpha=i(E_{\alpha}-E_{-\alpha}) \quad \makebox{and} \quad \fr{m}_{\alpha}=\mathbb RA_\alpha \oplus \mathbb RB_\alpha.\end{equation*}

 Then 
\begin{equation}\label{tangent}\fr{m}=\sum_{\alpha \in R_M^+}{\fr{m}_\alpha},\end{equation}

\begin{equation}\label{tangent1}\fr{k}=\fr{h}\oplus \sum_{\alpha \in R_K^+}{\fr{m}_\alpha}. \end{equation}

  Moreover, by using properties of root decompositions of simple Lie algebras, we obtain that
\begin{equation}\label{root}[\fr{m}_\alpha,\fr{m}_\beta]\subset \left\{ \begin{array}{ll}  \fr{m}_{\alpha+\beta}\oplus \fr{m}_{\left|\alpha-\beta \right|}, \quad \mbox{if} \quad \alpha+\beta \in R \quad \mbox{or} \quad \alpha-\beta \in R, \\
\left\{ {0} \right\}, \quad \mbox{otherwise}.
\end{array}
\right. \end{equation}

 The decomposition (\ref{tangent}) of $\fr{m}$ is orthogonal with respect to $B$.  We will now separate the submodules $\fr{m}_{\alpha}$, $\alpha\in R_M^+$ into two groups $\fr{m}_1,\fr{m}_2$ so that $[\fr{m}_1,\fr{m}_2]\subset \fr{m}_1$.  We achieve this as follows.

Each $\alpha\in R_M^+$ can be expressed as
\begin{equation}\label{roots}\alpha=c^{\alpha}_1\alpha_1+\cdots +c^{\alpha}_{r}\alpha_{r}+c^{\alpha}_{r+1}\alpha_{r+1}+\cdots +c^{\alpha}_{\ell}\alpha_{\ell},\end{equation}
where $c^{\alpha}_i\geq 0$, $c^{\alpha}_i\in \mathbb Z$ for $i=1,\cdots ,r+\kappa$ and at least one of the $c_1^{\alpha},\dots ,c^{\alpha}_r$ is not zero.  Each $\alpha\in R_K^+$ can be expressed as

\begin{equation}\label{roots1}\alpha=c^{\alpha}_{r+1}\alpha_{r+1}+\cdots +c^{\alpha}_{\ell}\alpha_{\ell},\end{equation}
where $c^{\alpha}_i\geq 0$ and $c^{\alpha}_i\in \mathbb Z$ for $i=r+1,\cdots ,\ell$.  Let $\alpha_{i_0}$ be a  simple root of $\Pi_M$.  We set\\

\begin{equation*}R_{M_{1}}=\left\{ {\alpha\in R_M^+:\makebox{$c^{\alpha}_{i_0}$ is odd}}\right\}, \quad R_{M_{2}}=\left\{ {\alpha\in R_M^+:\makebox{$c^{\alpha}_{i_0}$ is even}}\right\},\end{equation*}

\noindent and let
\begin{equation*}\fr{m}_1=\sum_{\alpha\in R_{M_{1}}}{\fr{m}_\alpha} \qquad \fr{m}_2=\sum_{\alpha\in R_{M_{2}}}{\fr{m}_\alpha}.\end{equation*}
Then $R_M^+=R_{M_{1}}\cup R_{M_{2}}$, $R_{M_{1}}\cap R_{M_{2}}=\emptyset$.  By taking into account the decomposition (\ref{tangent}) we conclude that 
\begin{equation}\label{conc}\fr{m}=\fr{m}_1\oplus \fr{m}_2.\end{equation}   
Also, by taking into account the expressions (\ref{roots}) and (\ref{roots1}), it follows that if $\alpha\in R_K$ and $\beta\in R_{M_i}$, then $\alpha+\beta\in R_{M_i}$, $i=1,2$.  We conclude that the decomposition (\ref{conc}) is $\operatorname{Ad}(K)$-invariant.  Moreover, by using relation (\ref{root}) we obtain that

\begin{equation*}[\fr{m}_1,\fr{m}_2]=\sum_{\alpha\in R_{M_1}, \beta\in R_{M_2}}{\fr{m}_{\alpha+\beta}\oplus \fr{m}_{\left|\alpha-\beta\right|}}.\end{equation*}
We claim that $[\fr{m}_1,\fr{m}_2]\subset \fr{m}_1$.  
Indeed, if $\alpha \in R_{M_1}$, $\beta \in R_{M_2}$, then  the coefficients $c^{\alpha}_{i_0}$ and $c^{\beta}_{i_0}$ are odd and even respectively.
Moreover, we have that $c^{\alpha+\beta}_{i_0}=c^{\alpha}_{i_0}+c^{\beta}_{i_0}$ and $c^{\left|\alpha-\beta\right|}_{i_0}=|c^{\alpha}_{i_0}-c^{\beta}_{i_0}|$.  Therefore, if $\alpha\in R_{M_1}$ and $\beta\in R_{M_2}$ then $\alpha+\beta\in R_{M_1}$ and $\left|\alpha-\beta\right|\in R_{M_1}$, and the claim follows. 

To summarise,  for a generalised flag manifold $G/K$ of  a simple Lie group $G$ there exist 
$\operatorname{Ad}(K)$-invariant decompositions $\fr{m}=\fr{m}_1\oplus \fr{m}_2$ with respect to the negative of the Killing form of $\fr{g}$, such that $[\fr{m}_1,\fr{m}_2]\subset \fr{m}_1$.  Such decompositions are determined by a choice of $\alpha_{i_0}\in \Pi_M$, therefore the number of such decompositions is equal to (at least) the cardinality of $\Pi_M$.  Then by using Corollary \ref{criterion} we obtain the following.

\begin{prop} Let $G/K$ be generalized flag manifold of a simple Lie group $G$ and let $\Pi_M$ be the set of simple complementary roots of $G/K$.  Let $p$ be the cardinality of $\Pi_M$.  Then $G/K$ admits at least $p$ one-parameter families of Riemannian metrics $g_{\lambda_1},\dots ,g_{\lambda_p}$, such that each space $(G/K, g_{\lambda_i})$, $i=1,\dots,p$ is a two-step g.o. space.\end{prop}

We give some examples of the above construction.
 
\begin{example} Let $G/K$ be a generalized flag manifold whose isotropy representation decomposes into two irreducible submodules $\fr{m}_1,\fr{m}_2$.  The submodules $\fr{m}_1,\fr{m}_2$ constructed above, are obtained by choosing a simple root of Dynkin mark $2$ of $\fr{g}^{\mathbb C}$.  Note that the metric of $G/K$ corresponding to the inner product $\langle \ ,\ \rangle=B|_{\fr{m}_1}+2B|_{\fr{m}_2}$ is a K\"ahler-Einstein metric (cf. \cite{BH}).
\end{example}

\begin{example} Let $G/K$ be a generalized flag manifold with three isotropy summands.  This is constructed by choosing either $\Pi_M=\left\{ {\alpha_{i_0}} \right\}$ with Dynkin mark $\mu_{i_0}=3$ (Type I), or  $\Pi_M=\left\{ {\alpha_1,\alpha_2} \right\}$ with Dynkin marks $\mu_1=\mu_2=1$ (Type II).  The classification of flag manifolds with three isotropy summands was obtained in \cite{Ki}.\\
Here we consider  a flag manifold $G/K$ of Type I. 
Let $\alpha\in R_M^+$ and let $c^\alpha_{i_0}$ be the coefficient of $\alpha_{i_0}$ in the decomposition (\ref{roots}).
The three isotropy submodules $\fr{n}_1, \fr{n}_2, \fr{n}_3$ $T_o(G/K)$ are given by 

\begin{equation*}\fr{n}_i=\sum_{\alpha\in R_M^+: c^\alpha_{i_0}=i}{\fr{m}_{\alpha}},\qquad i=1,2,3.\end{equation*}

We set $\fr{m}_1=\fr{n}_1\oplus \fr{n}_3$ and $\fr{m}_2=\fr{n}_2$.  Then, it is easy to check that $[\fr{m}_1,\fr{m}_2]\subset \fr{m}_1$, hence $G/K$ endowed with the metric $g_{\lambda}$ induced by the inner product $\langle \ ,\ \rangle=\left.B\right|_{\fr{m}_1}+\lambda\left.B\right|_{\fr{m}_2}$ is a Riemannian two-step g.o. space.
\end{example}

\subsection{Generalized Wallach spaces}

Let $M=G/K$ be a generalized Wallach space. Here $G$ is a compact semisimple Lie group and the isotropy representation of $M$ decomposes into three irreducible submodules $\fr{n}_i$, $(i=1,2,3)$ such that 
\begin{equation}\label{ni}[\fr{n}_i,\fr{n}_i]\subset \fr{k}.\end{equation}
  The classification of generalized Wallach spaces was recently obtained in \cite{Ni} and in
  \cite{Ch-Ka-Li}.  In \cite{Ar-So}  the authors investigated $G$-invariant metrics whose geodesics are of the form $\gamma(t)=\pi(\exp(tX) \exp(tY) \exp(tZ))$, $X,Y,Z\in \fr{m}$. 
  The main result there (Theorem 1.2) states that if $M$ is endowed with a $G$-invariant metric $g_{\lambda}$ induced by the inner product 
\begin{equation}\label{wal}\langle \ ,\ \rangle=\left.B\right|_{{\fr{n}}_i+\fr{n}_j}+\lambda \left.B\right|_{\fr{n}_l}, \quad i,j,l \quad \makebox{pairwise distinct}\end{equation}
\noindent
(where $B$ is the negative of the Killing form of $\fr{g}$), 
then the geodesics $\gamma$ of $(M,g_{\lambda})$ with $\gamma(0)=o$ are of the form $\gamma(t)=\pi(\exp(tX)\exp(tY))$, $X,Y\in \fr{m}$.  
If we set $\fr{m}_1=\fr{n}_i\oplus \fr{n}_j$ and $\fr{m}_2=\fr{n}_l$ we see that $\fr{m}= \fr{m}_1\oplus \fr{m}_2$.  Moreover, relation (\ref{ni}) implies that $[\fr{m}_1,\fr{m}_2]\subset \fr{m}_1$.
By using Corollary \ref{criterion} it follows that every generalized Wallach space $M$ admits three distinct one-parameter families of metrics $g_{\lambda}$, induced from the inner products (\ref{wal}), such that $(M,g_{\lambda})$ is a 2-step g.o space. This verifies the main theorem in \cite{Ar-So}.

\subsection{$k$-symmetric spaces, $k$ even}

Let $G$ be a connected Lie group and let $\phi:G\rightarrow G$ be an automorphism of $G$ such that $\phi^k=\operatorname{id}$.  Let $G^{\phi}=\left\{ {g\in G:\phi(g)=g} \right\}$ be the subgroup of fixed points of $G$ and let $G_o^{\phi}$ be its identity component.  Assume that there exists a closed subgroup $K$ of $G$ such that $G_o^{\phi}\subset K\subset  G^{\phi}.$  
Then $G/K$ is called a $k$-{\it symmetric space} (cf. \cite{Ko1}, \cite{Ko2}).
Let $G/K$ be a $k$-symmetric space where $G$ is a compact, connected and semisimple Lie group.  Let $B$ be the negative of the Killing form of $\fr{g}$.  Let $s=\left.[\frac{k-1}{2}\right]$, the integer part of $\frac{k-1}{2}$.  We set

\begin{equation*}u=\left\{ 
\begin{array}{lll}
s,  \quad \makebox{if $k$ is odd} \ \\

s+1 ,   \quad \makebox{if $k$ is even}.
\end{array}
\right.
\end{equation*}

We obtain the following $\operatorname{Ad}(K)$-invariant and orthogonal decomposition of $\fr{g}$ with respect to $B$ (cf. \cite{Ba},\cite{Ba-Sa}):

\begin{equation*}\fr{g}=\fr{k}\oplus \fr{m}=\fr{n}_0\oplus \fr{n}_1\oplus \cdots \oplus \fr{n}_t,\end{equation*}

where $\fr{n}_0=\fr{k}$ and the subspaces $\fr{n}_i$ satisfy the following relations for $0 \leq j \leq i \leq t$:

\begin{equation}\label{hat}[\fr{n}_i,\fr{n}_j]=\left\{ 
\begin{array}{lll}
\fr{n}_{i+j}\oplus \fr{n}_{i-j},  \quad i+j\leq t \ \\

\fr{n}_{k-(i+j)}\oplus \fr{n}_{i-j} ,   \quad i+j>t.
\end{array}
\right.
\end{equation}

We assume that $G/K$ is a $k$-symmetric space with $k$ even and we set

\begin{equation*} \fr{m}_1=\sum_{1\leq 2i+1 \leq t}{\fr{n}_{2i+1}} \qquad \fr{m}_2=\sum_{1\leq 2i \leq t}{\fr{n}_{2i}}.\end{equation*}

Then we have that $\fr{m}=\fr{m}_1\oplus \fr{m}_2$.  Moreover, by using relations (\ref{hat}) it follows that $[\fr{m}_1,\fr{m}_2]\subset \fr{m}_1$.  By using Corollary \ref{criterion} we obtain the following proposition:

\begin{prop}Let $G/K$ be a homogeneous symmetric space of order $2k$, where $G$ is a compact, semisimple and connected Lie group.  Then $G/K$ admits a family of metrics $g_{\lambda}$ such that $(G/K,g_{\lambda})$ is a two-step g.o space.\end{prop}

\end{document}